\documentclass[11pt]{article}

\def\A{\mathcal{A}_W} 
\def\B{\mathcal{A}_B} 
\def\G{\mathcal{G}_{W,B}} 
\def\M{\mathcal{M}_W} 
\def\MB{\mathcal{M}_B} 
\def\NW{\mathcal{N}_W^*} 
\def\NB{\mathcal{N}_B^*} 

\usepackage{amsmath,amsthm,amssymb,epsfig}

\newtheorem{theorem}{Theorem}
\newtheorem{lemma}{Lemma}
\newtheorem{corollary}{Corollary}
\newtheorem{conjecture}{Conjecture}
\newtheorem{definition}{Definition}

\def\twid{\char"7E }

\newtheorem{example}{Example}[section]

\def\clap#1{\hbox to 0pt{\hss #1\hss}}

\providecommand{\boxed}[1]{\clap{\framebox{\phantom{0}}}#1}

\providecommand{\looped}[1]{\clap{\unitlength1ex\begin{picture}(0,0)    
\put(0.6,0.8){\circle{2.8}}
\end{picture}}#1}   

\def\sspacer{\vline depth 6pt height 0pt width 0pt}
\def\nspacer{\vline depth 0pt height 15pt width 0pt}

\def\eref#1{$(\ref{#1})$}
\def\egref#1{Example~$\ref{#1}$}
\def\sref#1{\S$\ref{#1}$}
\def\dref#1{Definition~$\ref{#1}$}
\def\lref#1{Lemma~$\ref{#1}$}
\def\tref#1{Theorem~$\ref{#1}$}
\def\fref#1{Figure~$\ref{#1}$}
\def\cref#1{Corollary~$\ref{#1}$}
\def\cjref#1{Conjecture~$\ref{#1}$}

\def\per{{\rm per}}

\def\Z{\mathbb{Z}}

\def\V{\mathcal{V}}
\def\<{\big\langle}
\def\>{\big\rangle}
\def\iso{\cong}
\def\id{\varepsilon}

\begin{document}

\title{Latin trades in groups defined on \\
planar triangulations
}

\author{Nicholas J.\ Cavenagh and Ian M.\ Wanless\thanks{Research supported 
by ARC grant DP0662946.} \\
\small School of Mathematical Sciences\\[-0.5ex]
\small Monash University\\[-0.5ex]
\small Vic 3800, Australia\\
\small\tt \{nicholas.cavenagh,ian.wanless\} \ @sci.monash.edu.au}

\date{}

\maketitle

\begin{abstract}
For a finite triangulation of the plane with faces properly coloured 
white and black, let $\A$
be the abelian group constructed by labelling the vertices with
commuting indeterminates and adding relations which say that the
labels around each white triangle add to the identity. We show that
$\A$ has free rank exactly two. Let $\A^*$ be the torsion subgroup of
$\A$, and $\B^*$ the corresponding group for the black triangles. We
show that $\A^*$ and $\B^*$ have the same order, and conjecture that
they are isomorphic.

For each spherical latin trade $W\!$, we show there is a unique disjoint
mate $B$ such that $(W,B)$ is a connected and separated bitrade.  The
bitrade $(W,B)$ is associated with a two-colourable planar
triangulation and we show that $W$ can be embedded in $\A^*$, thereby
proving a conjecture due to Cavenagh and Dr\'apal. The proof involves
constructing a $(0,1)$ presentation matrix whose permanent and
determinant agree up to sign. The Smith Normal Form of this matrix
determines $\A^*$, so there is an efficient algorithm to construct the
embedding.  Contrasting with the spherical case, for each genus
$g\ge1$ we construct a latin trade which is not embeddable in any
group and another that is embeddable in a cyclic group.

We construct a sequence of spherical latin trades which cannot be
embedded in any family of abelian groups whose torsion ranks are
bounded. Also, we show that any trade that can be embedded in a
finitely generated abelian group can be embedded in a finite abelian
group. As a corollary, no trade can be embedded in a free abelian
group.

\end{abstract}

\section{Introduction}\label{s:intro} 

This paper demonstrates connections between 2-colourable
triangulations of the plane, latin trades, matrices whose permanent
and determinant agree up to sign, Smith normal forms and finite
abelian groups. It can thus be seen to link topological graph theory,
combinatorial designs, linear algebra and group theory.

Latin trades describe the difference between two latin squares of the
same order. Such differences are implicitly observed in many different
papers, some of them very old. The explicit theory of latin trades
arose from a shift in perspective, when the set of differences became
an object in its own right.  In this sense the earliest article on
latin trades is \cite{DrKe1}, where they are referred to as {\em
  exchangeable partial groupoids}.  Later, latin trades became of
interest to researchers of critical sets (minimal defining sets of
latin squares) \cite{BaRe2},\cite{codose},\cite{Ke2}.  Recently, the
theory of latin trades has found intersection with topology
\cite{group2}, geometry \cite{Dr1} and group theory \cite{CaHaDr}.  In
particular, latin trades may be interpreted topologically and under
certain structural conditions, assigned an integer genus.

In a sense we will make precise in \sref{s:embed}, a latin
trade $W$ is embeddable in a group $G$ if you can find a copy of $W$
in the Cayley table of $G$. In certain applications (e.g. \cite{EW,WW})
it is useful to find a latin square which differs only slightly from
a group. Finding such examples is equivalent to finding a small 
trade embedded in the group.

Dr\'apal \cite{Dr1} showed that certain partitionings of integer-sided
equilateral triangles into smaller ones imply the existence of latin
trades of genus $0$ which embed in the group $\Z_n$.  Cavenagh
\cite{Ca4} studied embeddings of $3$-homogeneous latin trades (a type
of latin trade with genus $1$) into abelian groups.  These results
motivate the question of whether there is any relationship between the
genus of a latin trade and its embeddability into a group.
\lref{l:doesembed} and \lref{l:doesnotembed} combine to show:

\begin{theorem}\label{t:bothhappen}
For every $g\ge1$ there is a latin trade of genus $g$ which
cannot be embedded in any group, and another latin trade of genus
$g$ which can be embedded in a cyclic group.
\end{theorem}

By contrast, the trades of genus zero can all be embedded in groups.
A major result of this paper is:

\begin{theorem}\label{t:main}
Every spherical latin trade can be embedded in a finite abelian group.
\end{theorem}

Moreover, there is an efficient algorithm for constructing what
is in some sense the canonical (though not necessarily the minimal)
such embedding. As indicated in the abstract, we also establish various 
properties of spherical trades and the groups that trades can be
embedded in.

\tref{t:bothhappen} and \tref{t:main} solve Open problem 8 from
\cite{pragueQs} (Proposed by Ale\v s Dr\'apal and Nick Cavenagh),
which asked: ``Can every separated latin trade be embedded into the
operation table of an abelian group? If this is not true in general is
it true for spherical latin trades?''

\section{Definitions}\label{s:def} 
  
The following definitions and a more comprehensive survey on 
latin bitrades may be found in \cite{Ca07}. 
 
A {\em latin square} $L$ of order $n$ is an $n\times n$ array, with
the cells of the array occupied by elements of a set $S$ of $n$ {\em
symbols}, such that each symbol occurs precisely once in each row and
once in each column.  If we also index the rows and columns of $L$ by
the sets $R$ and $C$, respectively, then $L$ may be thought of as a
subset of $R\times C\times S$.  Specifically, $(r,c,s)\in L$ if and
only if symbol $s$ occurs in row $r$ and column $c$ of the latin
square.

It is clear that a Cayley table for a group is a latin
square. However, there are latin squares which do not correspond to
group tables.  In general, a latin square is equivalent to an
operation table for a {\em quasigroup}. A quasigroup is a set of order
$n$ with binary operation $\star$ such that for each $a$ and $b\in Q$,
there exist unique elements $x$ and $y\in Q$ such that $a\star x=b$
and $y\star a=b$.

A {\em partial latin square} (PLS) of order $n$ is an $n\times n$
array, possibly with empty cells, such that each symbol from $S$
occurs {\em at most} once in each row and at most once in each column.
Thus any subset of a latin square $L$ is a PLS. The converse, however,
is not true, as some PLS have no completion to latin squares of the
same order.

Two PLS are said to be {\em isotopic} if one may be
obtained from the other by relabelling within the sets $R$, $C$ and $S$.
Two PLS are said to be {\em conjugate} to each other if one may be
obtained from the other by considering them as subsets of $R\times
C\times S$ and permuting the coordinates in their triples.
Two PLS are said to be in the same {\em main class} or {\em species}
if one may be obtained from the other by a combination of isotopy
and conjugation.

We next define latin bitrades (several equivalent definitions may
be found in~\cite{Ca07}).
 
\begin{definition}\label{d:bitrade2}
A latin bitrade is a pair $(W,B)$ of non-empty PLS
such that for each $(r,c,s)\in W$ (respectively, $B$), there
exists unique $r'\neq r$, $c'\neq c$ and $s'\neq s$ such that:
\begin{itemize}
\item $(r',c,s)\in B$ (respectively, $W$),
\item $(r,c',s)\in B$, (respectively, $W$), and
\item $(r,c,s')\in B$, (respectively, $W$).
\end{itemize}
\label{def2}
\end{definition}
The {\em size} of $(W,B)$ is equal to $|W|$; i.e. the number of filled
cells in $W$ (or, equivalently, in $B$). We stress that, in
this paper, all PLS (and in particular all trades)
are by definition finite.

If $(W,B)$ is a latin bitrade, we may refer to $W$ as a {\em latin
trade} and $B$ as its {\em disjoint mate}.  Equivalently, a {\em
latin trade} is a partial latin square $W$ for which there exists a
disjoint mate $B$ such that $(W,B)$ is a latin bitrade. It is possible
that a latin trade may have more than one choice of disjoint mate,
although we will see in \lref{l:uniqmate} an important case where
there is a unique mate satisfying certain extra conditions.

It is immediate from \dref{def2} that any isotopy or conjugate of
a latin bitrade is also a latin bitrade and that $(W,B)$ is a latin
bitrade if and only if $(B,W)$ is a latin bitrade.  If $(W,B)$ is a
latin bitrade, it is not hard to see that each non-empty row and each
non-empty column must contain at least two symbols, and that each
symbol occurs at least twice (if at all) within $W$ and $B$.
The smallest possible size of a latin bitrade is $4$; such latin
bitrades correspond to latin subsquares of order $2$ and are called
{\em intercalates}.  A latin bitrade $(W,B)$ is said to be {\em
  connected} if there exists no latin bitrades $(W',B')$ and $(W'',B'')$
such that $W'\cap W''=\emptyset$, $W=W'\cup W''$ and $B=B'\cup B''$.

\subsection{The genus of separated, connected latin bitrades}\label{ss:genus}

In this section we exclude from $R\cup C\cup S$ any rows, columns or
symbols that are not used within any triple of a particular latin
bitrade and we also specify that $R$, $C$ and $S$ are pairwise
disjoint.

Each row $r$ of a latin bitrade $(W,B)$ defines a
permutation $\psi_r$ of the symbols in row $r$, where $\psi_r(s)=s'$
if and only if $(r,c,s)\in W$ and $(r,c,s')\in B$ for some $c$. 
If $\psi_r$ is a single cycle, then we say that
the row $r$ is {\em separated}.  Similarly, we may classify each
column and symbol as being either separated or non-separated.  If
every row, column and symbol is separated, then we say that the
latin bitrade $(W,B)$ is {\em separated}.

Any non-separated latin bitrade may be transformed into a separated
latin bitrade by a process of identifying each cycle in the
permutation with a new row, column or symbol, as shown in the
following example.
\begin{example}
In the diagram below, subscripts denote symbols from disjoint mates. 
The first row $r$ of the latin bitrade $(W,B)$ is non-separated since
$\psi_r=(14)(35)$. The separated latin bitrade $(W',B')$ 
is formed by splitting $r$ into two rows. 
$$
\begin{array}{|ccccc|}
\hline
\cdot& 1_4 & 3_5 & 4_1 & 5_3 \\
1_2 & 3_1 & 2_3 & \cdot & \cdot \\
2_1 & 4_3 & 5_2 & 1_4 & 3_5 \\
\hline
\multicolumn{5}{c}{(W,B)\vline width 0pt height 14pt}
\end{array} 
\quad \quad  
\begin{array}{|ccccc|}
\hline
\cdot & 1_4 & \cdot  & 4_1 & \cdot \\
\cdot & \cdot & 3_5 & \cdot & 5_3 \\
1_2 & 3_1 & 2_3 & \cdot & \cdot \\
2_1 & 4_3 & 5_2 & 1_4 & 3_5 \\
\hline
\multicolumn{5}{c}{(W',B')\vline width 0pt height 14pt}
\end{array} 
$$
\label{sepa}
\end{example}
Note that the process of separation does not change
the size of the latin bitrade, or whether or not it is connected.

Given a separated, connected latin bitrade $(W,B)$, we may construct
a {\em triangulation} $\G$ whose vertex set is $R\cup C\cup S$ and whose
edges are pairs of vertices that occur within some triple of $W$ (or,
equivalently, some triple of $B$).  If we define white and black
faces of $\G$ to be the triples from $W$ and $B$ (respectively), then
$\G$ is a face $2$-colourable triangulation of some surface (\cite{group2}).

We next orient the edges of $\G$ so that each white face contains
directed edges from a row to a column vertex, a column to a symbol
vertex and a symbol to a row vertex.  It is immediate that each face
(black or white) is labelled {\em coherently}; i.e.  in the subgraph
induced by the vertices of a triangular face, both the out-degree and
the in-degree of each vertex are $1$.  Thus the surface in which $\G$
is embedded is {\em orientable}, and Euler's genus formula must give a
non-negative, integer value for the genus $g$:
\begin{align}
g & = \tfrac12(2-f+e-v) \nonumber \\
 & =  \tfrac12(2-|W|-|B|+3|W|-|R|-|C|-|S|)\nonumber \\
& =   \tfrac12(2+|W|-|R|-|C|-|S|),
\label{e:genus}
\end{align} 
where $f$, $e$ and $v$ are the number of faces, edges and vertices of
$\G$, respectively. We note that the graph Dr\'apal uses in \cite{group2}
to calculate the genus is based on the dual of $\G$.

We stress that it only makes sense to calculate the genus of a bitrade
if it is separated and connected. Hence, for the remainder of the
paper, whenever we ascribe a genus to a bitrade we mean this to imply
that the bitrade is separated and connected.

\begin{example}
\fref{f:interc} shows an intercalate latin bitrade $(W,B)$ and
its corresponding triangulation $\G$ embedded in the plane.  The
triangles corresponding to $B$ are shaded. The triangle
$(r_0,c_0,s_0)$ is the external face.
\end{example}

\begin{figure}[h]
\centerline{\includegraphics[height=150pt]{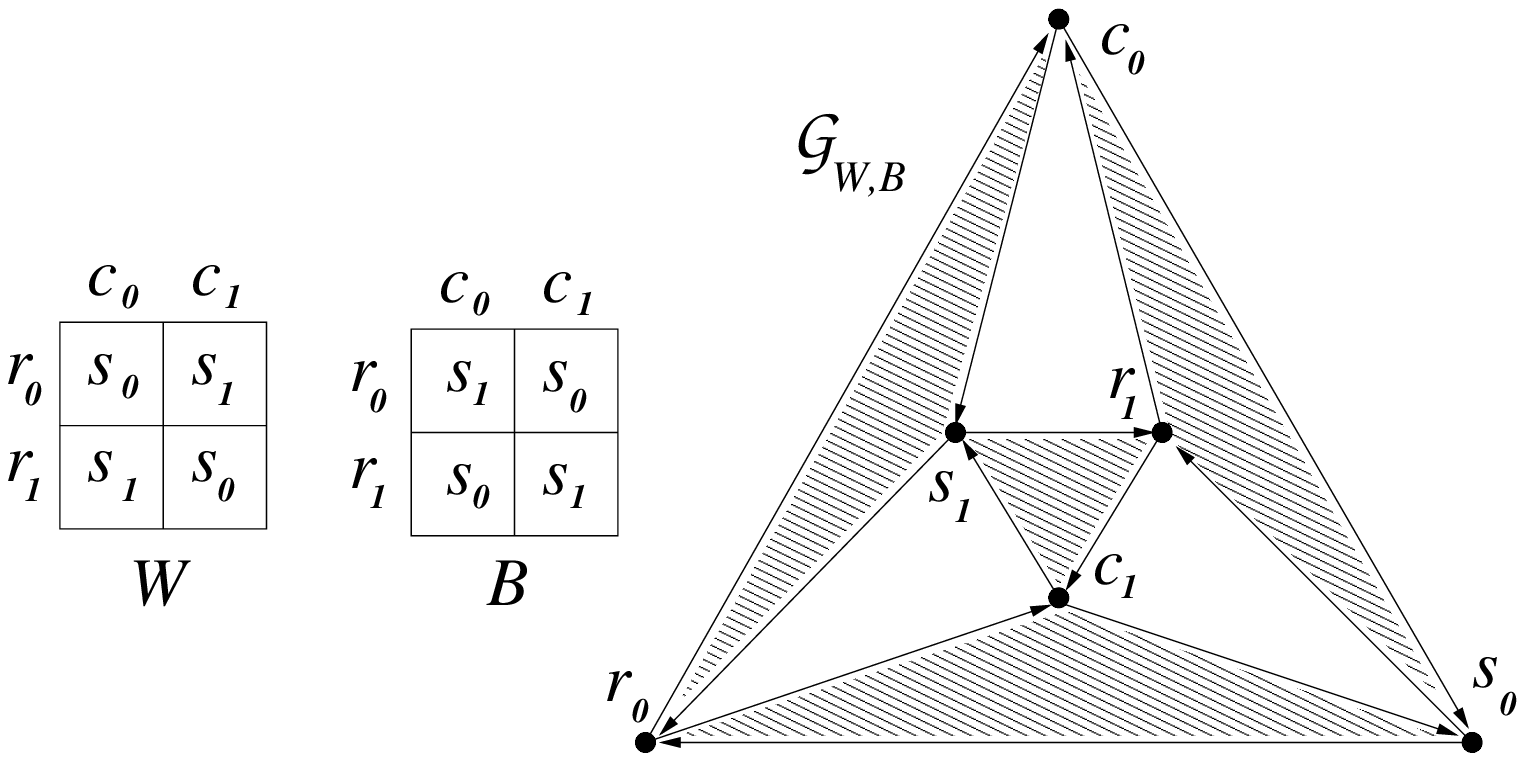}}
\caption{\label{f:interc}Intercalates $(W,B)$ with 
the corresponding triangulation $\G$.} 
\end{figure}

A latin bitrade of genus $0$ is called a {\em planar} or {\em
  spherical} latin bitrade.  We say that a latin {\em trade} $W$ is
{\em planar} or {\em spherical} if there exists some disjoint mate $B$
such that the latin bitrade $(W,B)$ is planar.  We will see later, in
\lref{l:uniqmate}, that such a choice of $B$ is unique, although there
may be a trade $B'\ne B$ such that $(W,B')$ is a latin bitrade but
is not separated and not connected.

Cavenagh and Lison\v ek demonstrated an
interesting equivalence involving planar latin bitrades.  A planar
Eulerian triangulation is an Eulerian graph embedded in
the plane with each face a triangle. In the
following theorem the latin bitrades are unordered; meaning $(W,B)$ is
considered to be equivalent to $(B,W)$.
\begin{theorem}{\rm (\cite{Cali})}
Planar embeddings of {\em unordered} latin bitrades are equivalent to 
planar Eulerian triangulations.
\label{pla}
\end{theorem} 

\section{Embeddings}
\label{s:embed}

A partial latin square $P$ is {\em embedded} in a
quasigroup $Q$ with operation $\star$ if $R,C,S\subseteq Q$ and
$r\star c=s$ for each $(r,c,s)\in P$.  We say that a PLS $P$ {\em
embeds} in, or {\em is embeddable} in, a quasigroup $Q$ if it is
isotopic to a PLS which is embedded in $Q$.  If $Q$ is a group, this
property is species invariant, as shown in the following lemma.

\begin{lemma}\label{l:specinvar}
Let $G$ be an arbitrary group and $P,P'$ two PLS from the 
same species. Then $P$ embeds in $G$ iff $P'$ embeds in $G$.
\end{lemma}

\begin{proof}
By definition, embedding is isotopy invariant, so it suffices to
consider the case where $P$ and $P'$ are conjugate.  This case
follows directly from Theorem 4.2.2 of \cite{DKI}, which states that
the six conjugates of the Cayley table of a group are isotopic to each
other. So if $P$ embeds in $G$ we take conjugates of the resulting
embedding to yield embeddings for the different conjugates of $P$.
\end{proof}

\lref{l:specinvar} is not true, in general, for embeddings in Latin
squares other than group tables. In the following example, $P$ embeds
in $L$ as shown by the entries in {\bf bold}, 
but the transpose of $P$ cannot be embedded in $L$.
\[
P=\begin{array}{|ccc|}
\hline
0&1&2\\
1&2&0\\
\hline
\end{array}
\qquad
L=\begin{array}{|cccccc|}
\hline
0&1&2&3&4&5\\[0.5ex]
\bf1&0&3&\bf2&\bf5&4\\[0.5ex]
\bf2&4&0&\bf5&\bf1&3\\[0.5ex]
3&5&1&4&2&0\\[0.5ex]
4&3&5&1&0&2\\[0.5ex]
5&2&4&0&3&1\\
\hline
\end{array}
\]

Our second lemma shows that embeddability in a group
can be tested by considering a restricted class of isotopies.

\begin{lemma}\label{l:000}
Let $P$ be a PLS and suppose that $(r,c,s)\in P$.
Let $G$ be a group with identity element $\id$.  If $P$ embeds in $G$
then $P$ is isotopic to a PLS in $G$ by an isotopy that maps each of
$r,c$ and $s$ to $\id$. 
\end{lemma}

\begin{proof}
Suppose $P$ is embedded in $G$. Define the permutations 
$I_1,I_2,I_3:G\mapsto G$
by $I_1(x)=r^{-1}x$,
$I_2(x)=xc^{-1}$ and
$I_3(x)=r^{-1}xc^{-1}$.
The isotopy $(I_1,I_2,I_3)$ maps $(r,c,s)$
to $(r^{-1}r,cc^{-1},r^{-1}rcc^{-1})=(\id,\id,\id)$.
\end{proof}

In this paper we are not concerned with counting the number of ways
that a trade can be embedded in a group. However, we make the
following observation in passing.  Given a trade $W$ and group $G$ it
is possible for $W$ to have ``essentially different'' embeddings $E_1$
and $E_2$ in $G$ in the sense that there is no isotopism that
preserves $G$ and maps $E_1$ to $E_2$. For example, the order two
subgroups of $\Z_2\oplus\Z_4$ generated by, respectively, $(0,2)$ and
$(1,0)$ are ``essentially different'' embeddings of intercalates.  
The former is contained in an embedding of $\Z_4$, while the
latter is not.

We remind the reader that not every latin square is a Cayley table for
a group. A test for whether a latin square is isotopic to a group 
is the {\em quadrangle criterion}.  A latin square $L$ satisfies
the quadrangle criterion if and only if for each
$r_1,r_2,r_1',r_2',c_1,c_2,c_1'c_2',s_1,s_2,s_3$ such that
$$(r_1,c_1,s_1),(r_1,c_2,s_2),(r_2,c_1,s_3),
(r_1',c_1',s_1),(r_1',c_2',s_2),(r_2',c_1',s_3)\in L,$$ the cells
$(r_2,c_2)$ and $(r_2',c_2')$ contain the same symbol.  A latin square
is isotopic to the Cayley table of a group if and only if it satisfies
the quadrangle criterion (see \cite{DKI} for a proof).

Observe that if a partial latin square $P$ contains filled cells that
fail the quadrangle criterion and if $P$ is a subset of a latin square
$L$, then $L$ also fails the quadrangle criterion.  Thus if a latin
trade fails the quadrangle criterion, it does not embed into any group
$G$.  However, the converse is not true, as shown by the following
example.
\begin{example}\label{eg:noquadr}
The following latin bitrade $(W,B)$ is separated and connected. It
does not fail the quadrangle criterion, as can easily be seen by
noting that no pair of columns have more than one symbol in common.
\[W=
\begin{array}{|ccccccccc|}
\hline
1&2&3&4&5&6&7&8&9\\
2&3&4&5&6&7&8&9&1\\
4&5&9&8&1&3&2&6&7\\
\hline
\end{array}
\qquad B=
\begin{array}{|ccccccccc|}
\hline
2&3&4&5&6&7&8&9&1\\
4&5&9&8&1&3&2&6&7\\
1&2&3&4&5&6&7&8&9\\
\hline
\end{array}
\]
However $W$ cannot be embedded into any group. To see this, take 
pairs of rows and interpret them as
permutations in two row format. Writing these in cycle format,
the first two rows give us $\pi_{12}=(123456789)$ whilst the
first and third rows give us $\pi_{13}=(148639725)$. If $W$ could
be embedded in a group then, applying a condition due to 
Suschkewitsch {\rm\cite{Sus}} related to the proof of Cayley's theorem,
we would find that $\pi_{12}$ and $\pi_{13}$ have to generate a 
group of order at most $9$. As this clearly is not the case, $W$
cannot be embedded in a group. 

Another way to produce a latin rectangle that cannot be embedded in
any group is to construct one which has rows $i$ and $j$ for which
$\pi_{ij}$ is not a regular permutation (a regular permutation is one
in which all cycles have the same length).
\end{example}

As the above example demonstrates, the quadrangle criterion cannot be
used to prove that a latin trade embeds in some group.  However, it is
useful for showing that a trade does not embed in any group:

\begin{example}\label{eg:nogrp}
The smallest connected, separated trade which does not embed in any group is
shown below, with its disjoint mate (which also does not embed in a
group).
\[
\begin{array}{|cccc|}
\hline\nspacer
0&\boxed1&\boxed2&3\\[0.6ex]
\looped1&\boxed{\looped2}&\boxed3&0\\[0.6ex]
\looped2&\looped0&1&\cdot\sspacer\\
\hline
\end{array}
\qquad
\begin{array}{|cccc|}
\hline\nspacer
1&\boxed2&\boxed3&0\\[0.6ex]
\looped2&\boxed{\looped0}&\boxed1&3\\[0.6ex]
\looped0&\looped1&2&\cdot\sspacer\\
\hline
\end{array}
\]
The fact that both trades fail the quadrangle criterion, and thus do
not embed in a group, is demonstrated by the two quadrangles marked
with $\boxed{\phantom{0}}$ and $\,\looped{\phantom{00}}$. The fact
that no smaller trade fails to embed in a group is guaranteed by
\tref{t:main}, since by {\rm\cite{tradenum}} the only smaller trades of
positive genus are isotopic to $\Z_3$ and hence (trivially) embeddable
in a group.
\end{example}

On the other hand, there are many classes of latin trades which do
embed in groups.  An early result by Dr\'apal \cite{Dr1} gave a class
of geometrically constructed latin trades which embed in $\Z_n$:
\begin{theorem} 
{\rm (Dr\'apal, $\cite{Dr1}$)} Let $m$ and $n$ be positive integers.
Suppose that we can partition an equilateral triangle of
side $n$ into $m$ smaller (integer-sided) equilateral triangles, such
that each vertex of a triangle occurs as the vertex of at most $3$ of
the smaller triangles. Then there exists a corresponding planar latin
bitrade $(W,B)$ of size $m$ such that $W$ embeds into $\Z_n$.
\label{tria}
\end{theorem}

The fact that the latin trades in the previous theorem are planar led
to the exploration of the relationship between the genus of a latin
trade and whether it embeds into a group. This is part of the
motivation for us to prove that any spherical latin trade is
embeddable into an abelian group (\tref{t:main}).

We spend the rest of this section putting this theorem into context by
exploring what happens when the conditions are relaxed or modified.
We first show that there is not an analogous result for any genus $g>0$.  
We then demonstrate an example of each of the following: a latin trade
which embeds in an abelian group but not in any cyclic group, a latin
trade which embeds in a non-abelian group but not in any abelian
group, and a latin trade which does not embed in any abelian
group even though it separates to produce a spherical trade.
Finally, we consider the case of embeddings into abelian groups which
are infinite but finitely generated.  We show that this case, in
essence, reduces to the finite one.

To see that planarity is a necessary condition for our result, we 
present the following two lemmas, which together prove \tref{t:bothhappen}.

\begin{lemma}\label{l:doesembed}
For each integer $g>0$, there exists a separated, connected latin
bitrade $(W,B)$ of genus $g$ such that $W$ embeds in $\Z_{2g+1}$.
\end{lemma}

\begin{proof}
To construct such a trade, let $W$ be the first three rows of
$\Z_{2g+1}$.  Then form $B$ by permuting the
rows of $W$ by shifting the top row to the bottom. 
Clearly $(W,B)$ is
a latin bitrade. Since each column has three symbols and each symbol
occurs three times, the columns and symbols are separated.

As each symbol $s$ is replaced by symbol $s+1$ (mod $2g+1$) in the first
two rows, these rows are separated. In the last row, $s$ is replaced by
$s-2$ (mod $2g+1$).  As $2g+1$ is odd, this row is also separated.
Thus we have a separated, connected latin bitrade, so its genus 
is well-defined. From \eref{e:genus}, the genus is
$$\tfrac12\big(2+|B|-|R|-|C|-|S|\big)
= \tfrac12\big(2+3(2g+1)-3-2(2g+1)\big)=g.$$
\end{proof}

\begin{lemma}\label{l:doesnotembed}
For each integer $g>0$, there exists a separated, connected latin
bitrade $(W,B)$ of genus $g$ such that $W$ does not embed in any
group.
\end{lemma}

\begin{proof}
\egref{eg:nogrp} suffices for $g=1$, so we suppose that $g\ge2$.
Let $S=\{0,1,2\}\times \{0,1,2,\dots ,2g\}\setminus \{(2,1),(2,2),(2,2g)\}$.
Define a partial latin square $W$ so that:
$$W = \big\{(2,1,4),(2,2,1),(2,2g,3)\big\}
\cup \big\{(i,j,i+j)\mid (i,j)\in S\big\},$$
with $i+j$ evaluated modulo $2g+1$. 
Next, form $B$ by shifting the top row of $W$ to the bottom.
For illustration, we show the pair $(W,B)$ when $g=3$:
\[W=
\begin{array}{|ccccccc|}
\hline
0&\boxed{1}&\boxed{2}&3&4&5&6\\
\looped{1}&\boxed{\looped{2}}&\boxed{3}&4&5&6&0\\
\looped{2}&\looped{4}&1&5&6&0&3\\
\hline
\end{array}
\qquad B=
\begin{array}{|ccccccc|}
\hline
1&2&3&4&5&6&0\\
2&4&1&5&6&0&3\\
0&1&2&3&4&5&6\\
\hline
\end{array}
\]
It is not hard to show that $(W,B)$ is a connected, separated latin
bitrade.  The calculation that $(W,B)$ has genus $g$ is the same as in
the previous lemma.  Moreover, $W$ cannot embed in a group as it fails
the quadrangle criterion, as shown by the quadrangles displayed in the
example above, which are present for all $g\ge2$.
\end{proof}

\subsection{Some contextual examples}\label{ss:eg}

In this subsection, we give some contextual examples.
In a number of cases we claim that the example we cite is the smallest
possible. All such claims are based on the enumeration of trades and
bitrades up to size 19, reported in \cite{tradenum} (the data is
downloadable from \cite{WWWW}). Note, by \lref{l:specinvar}, that
it suffices in all cases just to consider one representative of each
species.

In our next example, and a number of later examples, we 
establish that a given trade $W$ cannot be embedded in a certain type
of group. We do this by assuming it is embedded in such a group then
deriving a contradiction. To aid our calculations we adopt the
convention that the rows and columns, respectively, of $W$ are
labelled $r_0,r_1,r_2,\dots$ and $c_0,c_1,c_2,\dots$ in the order that
we give them in the example.  Also, if the example contains symbols
$0,1,2,\dots$, then we refer to these in our calculations as
$s_0,s_1,s_2,\dots$, respectively.

\begin{example}\label{eg:noncyc}
The smallest separated, connected bitrade $(W,B)$
such that $W$ cannot be embedded in any cyclic
group has size 10. The unique (up to interchanging $W$ and $B$)
example of that size is spherical. It is:
\[
W=\begin{array}{|cccc|}
\hline\nspacer
0&1&2&3\\[0.6ex]
1&0&\cdot&\cdot\\[0.6ex]
2&\cdot&1&\cdot\\[0.6ex]
\cdot&3&\cdot&0\sspacer\\
\hline
\end{array}
\qquad
B=
\begin{array}{|cccc|}
\hline\nspacer
2&3&1&0\\[0.6ex]
0&1&\cdot&\cdot\\[0.6ex]
1&\cdot&2&\cdot\\[0.6ex]
\cdot&0&\cdot&3\sspacer\\
\hline
\end{array}
\]
Suppose $W$ embeds in a cyclic group of order $m$ with operation
$\circ$ and identity $0$. 
Assuming, on the
basis of \lref{l:000}, that $r_0=c_0=s_0=0$ we then deduce that
$c_1=0\circ c_1=s_1=r_1\circ 0=r_1$ and $c_2=0\circ c_2=s_2=r_2\circ
0=r_2$.  Since $2c_1=r_1\circ c_1=0$ and $c_1\neq c_0=0$ we conclude
that $c_1=m/2$.  Now $2c_2=r_2\circ c_2=s_1=c_1$ so $c_2=\pm m/4$.
Also $r_3\circ c_3=0$ so $r_3=-c_3$, but $r_3\circ c_1=0\circ
c_3=c_3=-r_3$ so $2r_3=-c_1=m/2$. Thus $r_3=\pm m/4$. However this
means that $c_2$, $c_3$ and $r_3$ are each equal to $\pm m/4$. This is
impossible given that $r_3=r_3\circ c_0$ is distinct from both
$r_3\circ c_1=c_3$ and $r_2\circ c_0=c_2$, and $c_2$ and $c_3$ are
also distinct from each other.

A similar argument shows that $B$ cannot embed in a cyclic group of
order $m$. We deduce in turn that $r_0=c_0=s_2=0$, $r_2=c_2=m/2$,
$r_3\circ c_1=c_3$ and $r_3\circ c_3=c_1$. Combining these last two
equations we get $2r_3=0=2r_0=2r_2$, which contradicts the fact that
$r_0,r_2$ and $r_3$ are distinct.

Interestingly, although \tref{t:main} tells us that
$B$ can be embedded in an abelian group (in fact
both $B$ and $W$ embed in $\Z_2\oplus\Z_4$), the smallest group that $B$
embeds in is non-abelian. This follows from the above observation
that it does not embed in a cyclic group, the obvious fact that it does
not embed in $\Z_2\oplus\Z_2$, and the following embedding in $S_3$:
\[
\begin{array}{|cccccc|}
\hline\nspacer
\bf0&\bf1&\bf2&\bf3&4&5\\[0.5ex]
\bf1&0&3&\bf2&5&4\\[0.5ex]
\bf2&4&\bf0&5&1&3\\[0.5ex]
3&5&1&4&0&2\\[0.5ex]
4&2&5&0&3&1\\[0.5ex]
5&\bf3&4&\bf1&2&0\sspacer\\
\hline
\end{array}
\]
\end{example}

\begin{example}\label{eg:nonab}
The smallest connected separated bitrade $(W,B)$
in which $W$ embeds in a group but not in any abelian group has
size 14. An example of this size is:
\[
\begin{array}{|cccccc|}
\hline\nspacer
\looped{\bf0}&\boxed{\bf1}&\looped{\bf2}&3&\boxed{\bf4}&\bf5\\
\looped{\bf1}&\boxed{\bf0}&\looped4&5&\boxed{\bf2}&3\\
\bf2&\bf5&\bf0&4&3&1\\
3&4&5&0&1&2\\
4&3&1&2&5&0\\
\bf5&2&3&1&\bf0&\bf4\sspacer\\
\hline
\end{array}
\qquad
\begin{array}{|cccccc|}
\hline\nspacer
1&5&0&\cdot&2&4\\
2&1&\cdot&\cdot&0&\cdot\\
5&0&2&\cdot&\cdot&\cdot\\
\cdot&\cdot&\cdot&\cdot&\cdot&\cdot\\
\cdot&\cdot&\cdot&\cdot&\cdot&\cdot\\
0&\cdot&\cdot&\cdot&4&5\sspacer\\
\hline
\end{array}
\]
The left hand square is the non-abelian group $S_3$, in which the
trade $W$ is shown in {\bf bold}. The disjoint mate $B$ is given in
the right hand square.  To see that there is no smaller example, it
suffices by \tref{t:main} to check the $33$ smaller species of genus
$1$. These include \egref{eg:onegrp}, $6$ bitrades in which both trades
are embeddable in abelian groups and $26$ bitrades in which both trades
fail to embed in any group (as can be shown using the quadrangle
criterion). 

For the sake of contradiction, assume that $W$ can be embedded in an
abelian group with operation $\circ$ and identity $0$.  Assuming, on
the basis of \lref{l:000}, that $r_0=c_0=s_0=0$ we then deduce that
$c_1=0\circ c_1=s_1=r_1\circ 0=r_1$ and $c_2=0\circ c_2=s_2=r_2\circ
0=r_2$. Now observe that the two displayed quadrangles force $c_1\circ
r_2=r_1\circ c_2=s_4\neq s_5=r_2\circ c_1$, which violates the
assumption that $\circ$ is abelian.
\end{example}

\begin{example}\label{eg:nonsep}
The (implicit) condition in \tref{t:main} that the trade needs to be separated
cannot be abandoned. Consider the following example:
\[
\begin{array}{|cccc|}
\hline\nspacer
0&\boxed1&\boxed2&3\\
\looped1&\boxed{\looped2}&\boxed0&\cdot\\
\looped2&\looped3&\cdot&1\sspacer\\
\hline
\end{array}
\qquad
\begin{array}{|cccc|}
\hline\nspacer
\looped2&\boxed{\looped3}&0&\boxed1\\
\looped0&\looped1&2&\cdot\\
1&\boxed2&\cdot&\boxed3\sspacer\\
\hline
\end{array}
\]
This connected bitrade is not separated since $\psi_0=(02)(13)$,
although it may be separated into a spherical latin bitrade via the
process outlined in \sref{ss:genus}.  Both trades fail the quadrangle
criterion (as demonstrated by the quadrangles shown) and hence cannot
be embedded in any group.
\end{example}

\subsection{Embeddings in infinite groups}\label{ss:inf}

Next, we examine embeddings into infinite, not necessarily abelian, groups.
The general theme is that if a trade embeds in an infinite group then it
usually embeds in a related ``smaller'' group. We remind the reader that
in this paper trades are finite by definition.

\begin{lemma}\label{l:nathom}
Suppose $H$ is a normal subgroup of $G$ and $\phi$ is the natural
homomorphism $\phi:G\mapsto G/H$. Let $P$ be any PLS embedded in $G$.
If there exist $k_1,k_2\in G$ such that $\phi(x)=k_1H$ and $\phi(y)=k_2H$
for all triples $(x,y,z)\in P$, then $P$ embeds in $H$.
\end{lemma}

\begin{proof}
Consider the isotopy which maps each $(x,y,z)\in P$ to
$$(k_2^{-1}k_1^{-1}xk_2,k_2^{-1}y,k_2^{-1}k_1^{-1}z).$$  This is a well
defined isotopy since $k_1$ and $k_2$ are fixed.  It provides another
embedding of $P$ in $G$ because
\[
(k_2^{-1}k_1^{-1}xk_2)(k_2^{-1}y)
=k_2^{-1}k_1^{-1}xy
=k_2^{-1}k_1^{-1}z.
\]
Moreover, 
$\phi(k_2^{-1}k_1^{-1}xk_2)=\phi(k_1)^{-1}\phi(x)=H$
and $\phi(k_2^{-1}y)=\phi(k_2)^{-1}\phi(y)=H$ so we have an embedding of
$P$ in $H$.
\end{proof}

\begin{lemma}\label{l:dropZ}
Suppose $G$ is a group and $\phi:G\mapsto\Z$ a homomorphism.  Let
$(W,B)$ be a connected latin bitrade such that $W$ is
embedded in $G$.  Then $W$ embeds in the kernel of~$\phi$.
\end{lemma}

\begin{proof}
Define
\vskip-10pt\vskip-10pt
\begin{align*}
s&=\max\big\{\phi(z):(x,y,z)\in W\big\},\\
r&=\max\big\{\phi(x):(x,y,z)\in W,\ \phi(z)=s\big\},\\
R&=\big\{x:\exists(x,y,z)\in W,\ \phi(x)=r,\phi(z)=s\big\},\\
C&=\big\{y:\exists(x,y,z)\in W,\ \phi(x)=r,\phi(y)=s-r\big\}.
\end{align*}

Let $W'$ and $B'$ be the restriction of $W$ and $B$, respectively, to
the array defined by $R\times C$. 
Consider $(x,y,z)\in W$. Then $(x,y,z)\in W'$ if and
only if $\phi(x)=r$ and $\phi(z)=s$. On the other hand, consider
$(x,y',z)\in B$ such that $\phi(x)=r$ and $\phi(z)=s$.  By
\dref{d:bitrade2} there exist $x',z'$ such that $W$ contains the
triples $(x',y',z)$ and $(x,y',z')$.  By our choice of $r$,
$\phi(y')=\phi(z)-\phi(x')\ge s-r$ and by our choice of $s$,
$\phi(y')=\phi(z')-\phi(x)\le s-r$.  Hence $\phi(y')=s-r$ and $y'\in
C$. Thus if $(x,y',z)\in B$, $\phi(x)=r$ and $\phi(z)=s$, then
$(x,y',z)\in B'$.

By \dref{d:bitrade2}, for each $(x,y,z)\in W'$, there
is a unique $y'\neq y$ such that $(x,y',z)\in B$. From above, each such 
$(x,y',z)$ is in $B'$. Now, in each row of $R\times C$,  
$W'$ and $B'$ contain the same number of symbols, so for any $(x,y,z)\in B$
it follows that
$(x,y,z)\in B'$ if and only if  $\phi(x)=r$ and $\phi(z)=s$. 
Thus, using the notation from \sref{ss:genus}, for each $x\in R$, 
$\psi_x$ setwise fixes the set of symbols contained in columns from $C$. 

Next, suppose there is some $(x,y,z)\in B'$ and $(x',y,z)\in W$.  
Then $\phi(x)=r$ and $\phi(z)=s$ from above. Thus
$\phi(x')=\phi(z)-\phi(y)=r$, so $x'\in R$. Thus, for
each column $y\in C$, $\psi_y$ setwise fixes the set of symbols in
that column contained in rows from $R$.  Thus, the restriction of
$(W,B)$ to $R\times C$ is a bitrade.  Since $(W,B)$ is connected, it
contains no entries outside $R\times C$.  Finally, apply
\lref{l:nathom}.\end{proof}

In a finitely generated abelian group $G$, the {\em torsion subgroup}
$H$ is the subgroup consisting of all elements of finite order.  It is
well known that $G/H$ is a direct sum of finitely many copies
of~$\Z$. Thus by repeated use of \lref{l:dropZ} we have:

\begin{corollary}\label{c:torsion}
Let $(W,B)$ be a connected latin bitrade such that $W$ embeds
in a finitely generated abelian group $G$. Then $W$ embeds in the
torsion subgroup of $G$.
\end{corollary}

\begin{corollary}\label{c:fin}
Let $W$ be a trade. Then $W$ can be embedded in a finite
abelian group iff $W$ can be embedded in a finitely generated abelian group.
\end{corollary}

\begin{proof}
The ``only if'' direction is trivial.  For the ``if'' direction, apply
the previous corollary to each connected component of a latin bitrade
$(W,B)$, then take the direct sum of the resulting embeddings.
\end{proof}

Actually, a large finite cyclic group is locally indistinguishable
from $\Z$, so it is not hard to see that \cref{c:fin} is true for any
PLS, not just for trades. It suffices to replace each copy of $\Z$ by
$\Z_m$ where $m$ exceeds the maximum difference between any two
indices that the embedding uses in that copy of $\Z$.  However,
\cref{c:torsion} does not seem easy to obtain by such an argument, and
does not hold for all PLS (to see this, consider any PLS of size 5
embedded in $\Z_2\oplus\Z$).

A consequence of \cref{c:torsion} is that no trade can be
embedded in a torsion-free abelian group. However, it is possible to
embed a trade in a torsion-free non-abelian group, as we now show.

\begin{example}\label{eg:torsion}
Consider the group $G=\langle x,y\mid x^2yx^2=y,\ y^2xy^2=x\rangle$.
Let $z=xy$ and denote the identity element in $G$ by $\id$. Let $N$
be the subgroup of $G$ generated by $x^2$, $y^2$ and $z^2$.
The following relations are simple consequences of the relations for $G$.
\begin{equation}\label{e:normxyz}
\begin{matrix}
x^2y^2x^{-2}y^{-2}=x^2z^2x^{-2}z^{-2}=y^2z^2y^{-2}z^{-2}=\id\\
xy^2=y^{-2}x, \ xz^2=z^{-2}x, \ yx^2=x^{-2}y, \ zx^2=x^{-2}z, \ 
zy^2=y^{-2}z, \ yz^2=z^{-2}y. 
\end{matrix}
\end{equation}
It follows that $N$ is normal and abelian. By a
Reidemeister-Schreier rewriting process it can be shown that 
$N$ is a free abelian group of rank 3. Moreover,
$$G/N\iso\<x,y\mid x^2=y^2=(xy)^2=\id\>\iso\Z_2\oplus\Z_2.$$
Each element of $G$ may be written in the form 
$x^ay^b(x^2)^\alpha(y^2)^\beta(z^2)^\gamma$ where $a,b\in\{0,1\}$
and $\alpha,\beta,\gamma\in\Z$. Elements of this form may be
multiplied using the relations \eref{e:normxyz}. We find that
\begin{eqnarray*}
\big(x(x^2)^\alpha(y^2)^\beta(z^2)^\gamma\big)^2=x^{4\alpha+2},\\
\big(y(x^2)^\alpha(y^2)^\beta(z^2)^\gamma\big)^2=y^{4\beta+2},\\
\big(z(x^2)^\alpha(y^2)^\beta(z^2)^\gamma\big)^2=z^{4\gamma+2},
\end{eqnarray*}
each of which is a non-identity element of $N$. Hence the square of every
element of $G$ has infinite order, from which it follows
that $G$ is torsion-free.
Finally, we note that a spherical latin trade embeds in $G$:
$$\begin{array}{c|cccc}
&\id&yx&x&y^{-1}x^2\\
\hline
\id&\id&yx&x&\cdot\\
x^{-1}y^{-1}&x^{-1}y^{-1}&\id&\cdot&y^2x\\
x&x&\cdot&x^2&x^{-1}y^{-1}\\
y&\cdot&y^2x&yx&x^2\\
\end{array}
$$
\end{example}

\section{The main results}\label{s:main}

In this section we prove the main results of the paper, including 
\tref{t:main}. We start by outlining some notational conventions which
apply throughout the section.

We will use additive notation for groups to emphasise that they are
all abelian.
Throughout, $W\subseteq R\times C\times S$ 
is a latin trade of size $n+1$ with a fixed disjoint mate
$B$ such that $(W,B)$ is a spherical bitrade.  As in \sref{ss:genus},
we assume $R$, $C$ and $S$ are disjoint and we omit any unused rows,
columns and symbols.  We consider $\V=R\cup C\cup S$ to be a set of
{\em commuting} indeterminates.
From $W$ we define a group $\A$ with the 
following presentation:
\[\A=
\<\V\mid\{r+c+s=0:(r,c,s)\in W\}\>.
\]
Like all groups in this section, it is immediate from its 
definition that $\A$ is abelian. We now proceed to deduce some much
deeper properties of $\A$, which will eventually lead to a proof
of \tref{t:main}.

We start by considering the standard matrix presentation of $\A$.
We build $\M$, a $(0,1)$-matrix with rows indexed by $\V$ and columns
indexed by $W$, in which there are exactly three 1's per column. The
1's in the column indexed by $(r,c,s)\in W$ appear in the rows
indexed by $r$, $c$ and $s$. Since $(W,B)$ has genus $0$ and
size $n+1$, we know from \eref{e:genus} that $\M$ is an
$(n+3)\times(n+1)$ matrix.  In particular, it has more rows than
columns, which means that $\A$ is infinite. However, we can make use
of \lref{l:000} to create a related finite group which might help us
to prove \tref{t:main}.

In the triangulation $\G$, we consider the faces corresponding to
elements of $W$ to be {\em white} and faces corresponding to elements
of $B$ to be {\em black}.  We identify one white triangle as the {\em
special white triangle} $T_\Delta$; say
$T_\Delta=(r_0,c_0,s_0)$.  We then define a group $\A^*$ with the
following presentation:
\[\A^*=
\<\V\mid\{r+c+s=0:(r,c,s)\in W\}\cup \{r_0=c_0=s_0=0\}\>
\]
We also define a corresponding matrix presentation $\M^*$, which is
obtained from $\M$ by deleting the column corresponding to $T_\Delta$ and
the rows corresponding to $r_0$, $c_0$ and $s_0$. It follows that
$\M^*$ is an $n\times n$ matrix. On the face of it, $\A^*$ depends on
the choice of $T_\Delta$. However, this dependence is superficial,
as we show below.
 
\begin{lemma}\label{l:AAstar}
$\A\iso\A^*\oplus\Z\oplus\Z$.
\end{lemma}

\begin{proof}
Consider the homomorphism $\xi:\A\mapsto\A^*\oplus\Z^2$ satisfying
$\xi r=(r,1,0)$, $\xi c=(c,0,1)$ and $\xi s=(s,-1,-1)$, for every
$r\in R$, $c\in C$ and $s\in S$. It has an inverse satisfying
$\xi^{-1}(0,1,0)=r_0$, $\xi^{-1}(0,0,1)=c_0$, $\xi^{-1}(r,0,0)=r-r_0$,
$\xi^{-1}(c,0,0)=c-c_0$ and $\xi^{-1}(s,0,0)=s-s_0$, for every
$r\in R$, $c\in C$ and $s\in S$. Thus it provides the required isomorphism.
\end{proof}

\begin{corollary}\label{c:allwhite}
Let $\A^\#$ be defined as for $\A^*$ except with a different
choice of $T_\Delta$. Then $\A^\#\iso\A^*$. Moreover, the isomorphism
preserves $r-r'$ for all $r,r'\in R\subset\V$, the common generating
set for $\A^\#$ and $\A^*$.
\end{corollary}

\begin{proof}
The proof of \lref{l:AAstar} shows that $\A^\#\oplus\Z^2\iso\A^*\oplus\Z^2$
by an isomorphism that maps $(r-r',0,0)$ to $(r,1,0)-(r',1,0)=(r-r',0,0)$.
The corollary follows.
\end{proof}

Of course, the isomorphism in \cref{c:allwhite} also preserves $c-c'$
and $s-s'$ for all $c,c'\in C$ and $s,s'\in S$. We next prove a lemma
about triangles embedded in the plane.

\begin{lemma}\label{l:VinT}
Let $G$ be a simple plane graph.
Suppose that $T_1,T_2,\dots,T_k$ are distinct triangular faces of $G$,
such that every edge of $G$ belongs to precisely one of
$T_1,\dots,T_k$. Then $G$ has at least
$k+2$ vertices. If $k>1$ and $G$ has only $k+2$ vertices then every
vertex of $G$ is incident with at least two of $T_1,T_2,\dots,T_k$.
\end{lemma}

\begin{proof}
First suppose that $G$ is connected. 
We apply Euler's formula to $G$.
Let $U_1,\dots,U_m$ be the faces of $G$ other than $T_1,\dots,T_k$. 
From the given
conditions we know that $G$ has exactly $3k$ edges and every edge is
shared by some $T_i$ and some $U_j$. Since each $U_j$ has at least 3
edges, $3k\ge 3m$ and hence $k\ge m$.  Now Euler's formula tells us
the number of vertices in $G$ is $3k-(k+m)+2\ge k+2$.  If $G$ is not
connected, applying the above argument to each connected component
gives a similar result.

Moreover, the number of vertices in $G$ equals $k+2$ only if
$G$ is connected, $k=m$ and each $U_j$ is a triangle. Suppose $T_i$ has
vertex set $\{X,Y,Z\}$ where $Z$ is not a vertex of any $T_j$ for
$j\ne i$. Then the $U_l$ that contains $Z$ also contains $X$ and $Y$.
Assuming $k>1$, this leads to a parallel edge unless $U_l$ has more than 3
sides.
\end{proof}

Recall that $\M^*$ is a square matrix. We next prove that its permanent
is non-zero. Actually, we show something slightly stronger.

\begin{lemma}\label{l:altPM}
Let $\M'$ denote any matrix obtained from $\M^*$ by setting one of the
positive entries to zero. Then $\per(\M')>0$.
\end{lemma}

\begin{proof} 
Suppose by way of contradiction that $\per(\M')=0$. Since $\M'$ is a
non-negative matrix, the Frobenius-K\"onig Theorem (see, for example,
31.4.1 in \cite{WPer}) tells us that $\M'$ contains a $(n-k+1)\times k$
submatrix of zeroes for some $1\le k\le n$.  Let the column indices of
this submatrix be denoted $\Gamma$ and suppose these columns correspond to
the triangles $T_1,T_2,\dots,T_k$ in $\G$. As usual, let $T_\Delta$ denote the
special triangle.

Applying \lref{l:VinT} to $T_\Delta,T_1,\dots,T_k$ in
the subgraph of $\G$ induced by the edges of these triangles, we find that
$T_1,T_2,\dots,T_k$ between them use at least $k$ vertices that are
distinct from the 3 vertices in $T_\Delta$.  It follows that the columns
$\Gamma$ do not have more than $n-k$ zero rows in $\M^*$. Moreover,
the condition for equality in \lref{l:VinT} tells us there can only be
$n-k$ zero rows if each of the other rows has at least two positive
entries. Hence $\M'$ cannot contain $n-k+1$ zero rows in the columns
$\Gamma$. This contradicts the choice of $\Gamma$ and proves the
Lemma.
\end{proof}

\begin{corollary}\label{c:per2}
$\per(\M^*)\ge2$.
\end{corollary}

The lower bound in \cref{c:per2} is achieved when
$W$ is an intercalate.

An alternate way to state \lref{l:altPM} is that for any given edge in
the bipartite graph corresponding to $\M^*$ there is a perfect
matching that does not use that edge. The previous sentence is untrue
if the word ``not'' is omitted, as can be demonstrated by taking $W$
to be any trade of size 7.

There is a wealth of interesting theory surrounding matrices whose
permanent and determinant agree up to sign. The interested reader
is encouraged to start by consulting \cite{McC}. We use one small
part of that theory to prove our next result.

\begin{theorem}\label{t:detper}
$|\det(\M^*)|=\per(\M^*)$.
\end{theorem}

\begin{proof}
From \cref{c:per2} there exists at least one positive diagonal
$\delta$ of $\M^*$.  For a fixed choice of $\delta$ we construct a digraph
$D$ with vertices corresponding to the rows of $\M^*$.  For each
column $\beta$ of $\M^*$, let $\alpha$ be the row of $\M^*$ such that
cell $(\alpha,\beta)$ is included in $\delta$.  For each row
$\alpha'\neq\alpha$ that contains a non-zero entry in column $\beta$, we
add a directed edge $(\alpha,\alpha')$ to $D$.  

We will show that $D$ contains no dicycles of even length. 
By Theorem 10 of \cite{McC}, this gives the required result. (Note that
\cite{McC} uses a digraph isomorphic to our $D$ and a matrix that is
equivalent to our $\M^*$ up to a permutation of the rows. Such permutations
may change the sign of the determinant, but preserve its magnitude.)

Suppose that there exists a dicycle $\Gamma$ in $D$ of length $l\geq
3$.  If we ignore orientation then the vertices and edges of $D$ are a
subset of the vertices and edges, respectively, of $\G$.  Thus 
$\G$ induces a planar embedding of $D$.  By
redrawing $\G$ if necessary, we may assume that the
special triangle $T_\Delta$ lies outside $\Gamma$. Note that by
construction $\Gamma$ cannot contain any vertex from $T_\Delta$.  We let
$G'$ be the subgraph of $\G$ which includes $\Gamma$ and any internal
faces, vertices and edges.  Thus $G'$ has one external face with $l$
edges and all other faces are triangles. Suppose $G'$ has $v$ vertices,
$e$ edges and $f$ faces. Further, suppose that $y$ of the edges 
of $\Gamma$ are contained in white triangles in $G'$. These $y$ edges
belong to $y$ different white triangles in $G'$, because in any white
triangle the oriented edges share the same source, so a dicycle
cannot contain more than one of them.

Next, define $f_B$ to be the number of black triangles within $G'$ and
let $f_W$ be the number of white triangles within $G'$ that do not
intersect $\Gamma$.  Then, the (vertex,\,face) pairing given by $\delta$
implies that the number of vertices of $G'$ not on $\Gamma$ is equal
to $f_W$.  Thus, $v=l+f_W$.

Now, every edge of $G'$, other than the $y$ edges on $\Gamma$ which
border white triangles, is adjacent to a unique black
triangle. Therefore, $e = 3f_B+y$ and $f = 1 +f_B+f_W+y$.  Since $G'$
is planar, $f-e+v=2$ by Euler's formula. Combining these equations
yields $l=2(f_B-f_W)+1$, so in particular $l$ is odd.
\end{proof}

Combining \tref{t:detper} and \cref{c:per2} we see that
$|\det(\M^*)|\ge2$.  The theory of presentation matrices (see, for
example, \cite{Sims}) tells us that $|\det(\M^*)|$ is the order of $\A^*$
(and that $\det(\M^*)$ would be $0$ if $\A^*$ was infinite). Hence, we
have the vital observation that $\A^*$ is non-trivial and finite:

\begin{corollary}\label{c:Afinite}
The abelian group $\A^*$ satisfies $2\le|\A^*|<\frac{\sqrt2}{3}6^{n/3}$.
\end{corollary}

\begin{proof}
The matrix $\M$ has three 1's in every column. So of the $n$ columns
in $\M^*$, there are at least 3 with only two 1's and the remaining
columns contain three 1's.  Applying the Minc-Br\`egman Theorem (see,
for example, 31.3.5 in \cite{WPer}) gives the upper bound.
\end{proof}

Combined with \lref{l:AAstar} we also have:

\begin{corollary}
$\A$ has free rank exactly two.
\end{corollary}

We also now know that $\M^*$ is invertible. Our next technical lemma
tells us, in effect, that $(\M^*)^{-1}$ does not have any column which
consists of integers.

\def\vec{\tilde}

\begin{lemma}\label{l:cramer}
Let $\vec x=(x_1,x_2,\dots,x_n)^T$ and $\vec e_i=(0,\dots,0,1,0,\dots,0)^T$
where the sole $1$ is in the $i$-th coordinate.
The matrix equation $\M^*\,\vec x=\vec e_i$ has no solution
in which every $x_j$ is an integer.
\end{lemma}

\begin{proof}
Let $M_{ij}$ denote the cofactor of the $(i,j)$-th entry of $\M^*$.
From a cofactor expansion for $\det(\M^*)$ along
the $i$-th row we infer that
\begin{equation}\label{e:cofact}
1=\sum_{j\in\Omega}\frac{M_{ij}}{\det(\M^*)},
\end{equation}
where $\Omega$ is the set of columns in which a 1 appears in row $i$
of $\M^*$. By \lref{l:altPM} 
there are at least two $j\in\Omega$ for which $M_{ij}\neq0$. 
Also, by \tref{t:detper}, the non-zero terms in the sum in \eref{e:cofact} 
are all positive, so they cannot be integers.
The result now follows from Cramer's rule, since the unique solution to
$\M^*\,\vec x=\vec e_i$ is
$$x_j=\frac{M_{ij}}{\det(\M^*)}$$
for each $j$.
\end{proof}

Given \cref{c:Afinite}, our next result will complete the proof of
\tref{t:main}.

\begin{theorem}\label{t:embeds}
$W$ embeds in $\A^*$.
\end{theorem}

\begin{proof}
We need to distinguish between the elements of $\V$ when they are used
as coordinates for $W$ as contrasted with when they are considered as
generators for $\A^*$.  For any $v\in\V$ we will use $\bar v$
to denote the group element and $v$ without the bar to designate the
coordinate for $W$.

Consider the maps
\begin{align*}
I_1&:R\mapsto\A^*\hbox{ \ that satisfies \ }
     I_1(r_0)=0\hbox{ and }I_1(r_i)=\bar r_i\hbox{ otherwise},\\
I_2&:C\mapsto\A^*\hbox{ \ that satisfies \ }
     I_2(c_0)=0\hbox{ and }I_2(c_j)=\bar c_j\hbox{ otherwise},\\
I_3&:S\mapsto\A^*\hbox{ \ that satisfies \ }
     I_3(s_0)=0\hbox{ and }I_3(s_k)=-\bar s_k\hbox{ otherwise}.
\end{align*}
For any $(r_i,c_j,s_k)\in W$, the construction of $\A^*$ ensures
that $\bar r_i+\bar c_j+\bar s_k=0$ so that $I_1(r_i)+I_2(c_j)=I_3(s_k)$.
It follows that we have an embedding of $W$ in $\A^*$ provided we can
show that $I_1,I_2,I_3$ are injections.

We first prove that $\bar v\neq 0$ in $\A^*$ for any $v\in\V
\setminus\{r_0,c_0,s_0\}$. Suppose $\bar v=0$. Then $\bar v$ must be
an integer linear combination of the defining relations for
$\A^*$. However, \lref{l:cramer} shows this is not the case.

We next argue that $I_1(r)\neq I_1(r')$ for distinct $r,r'\in
R\setminus\{r_0\}$. Suppose $\bar r=\bar r'$. Consider the situation
if we had chosen, say, $(r',c,s)$ as the special white triangle rather
than $(r_0,c_0,s_0)$.  By \cref{c:allwhite} this would have produced
an isomorphic group in which $\bar r=\bar r'=0$, contrary to what we
proved above. It follows that $I_1$ is injective.

The proofs that $I_2$ and $I_3$ are injective are similar.
\end{proof}

Having shown that spherical trades can always be embedded in a finite
abelian group, we pause to consider some properties of the embedding
that we have constructed. By definition $\A$ is the ``most free''
abelian group that contains the structure of $W$ and is generated by
the elements in that structure. However, on the basis of \lref{l:000}
and \cref{c:torsion}, we feel it appropriate to concentrate on $\A^*$,
the torsion subgroup of $\A$.  We call the embedding of $W$ in $\A^*$
{\em canonical}.  Of course, $W$ can be embedded in any group that
contains $\A^*$ as a subgroup.  It is also plausible that $W$ could be
embedded in a group obtained by adding extra relations to $\A^*$, but
any such group is a homomorphic image of $\A^*$. In particular, $\A^*$
need not have the minimal order among groups in which $W$ embeds. We
will give a concrete example where the minimal and canonical
embeddings differ, in \egref{eg:diff}.

It is also worth mentioning that the embedding of $W$ in $\A^*$ can be
constructed in time polynomial in $|W|$.  We may not be able to
construct the whole Cayley table for $\A^*$, since that could
plausibly be exponentially big (see \cref{c:Afinite}).  However, we
can construct that part of the table which contains $W$.  It is a
simple matter to construct $\M^*$ from $W$.  We then find unimodular
matrices $U_1$ and $U_2$ such that $S_W=U_1\M^*U_2$ is the Smith
normal form of $\M^*$ (this can be done in polynomial time, see
\cite{Sims}).  Now $U_1$ encodes the linear relationship between the
old and new generators, and hence we can translate between the
(isomorphic) groups with presentation matrices $\M^*$ and $S_W$
respectively. From this information we get a transparent embedding of
$W$ in $\A^*$.

We next return to the order of the group $\A^*$.  The lower bound in
\cref{c:Afinite} is achieved by intercalates. The upper bound is
likely to be far from the truth.  It is clear from the proof of
\tref{t:detper} that $\M^*$ has a structure very different from the
matrices which achieve equality in the Minc-Br\`egman bound. However,
we have not made a concerted effort to improve the upper bound in
\cref{c:Afinite}.  It is worth noting, however, that every integer
$m\ge2$ is the order of $\A^*$ for some spherical latin trade $W$. In
particular, we can take $W$ to be the first two rows of $\Z_m$. It is
routine to check that $\A^*\iso\Z_m$, and that the
minimal and canonical embeddings coincide in this case.

Next, we consider {\em torsion rank}, which for abelian groups is
defined to be the minimum cardinality of a set of generators for the
torsion subgroup. We construct an infinite family of planar latin
trades $W$ such that if $W$ embeds in $G$ then the torsion rank of $G$
is at least logarithmic in $|W|$.  In particular, the torsion rank of
even the {\em minimal} embedding for $W$ can be forced to be arbitrarily
large.

\begin{lemma}\label{l:torank}
For each integer $m\geq 2$, there exists a 
planar latin bitrade $(W,B)$ of size $8m$ including distinct rows
$r,r_1,r_2,\dots ,r_m$ such that if $W$ is embedded in a finite abelian
group $G$ then $2r=2r_1=\dots =2r_m$.
\end{lemma}

\begin{proof}
Let $W_0$ be the latin trade formed from the first two rows of the
addition table for $\Z_m$, with disjoint mate $W_0'$ formed by
swapping these rows.
 
By observation, $(W_0,B_0)$ is a connected, separated planar latin
bitrade of size $2m$.  Let $r$ be a row of this latin bitrade.  In any
planar embedding, there are $m$ black triangles incident with $r$.
Let $(r,c,s)$ be one such black triangle.

Introducing rows $r',r_1$, columns $c',c''$ and symbols $s',s''$ which
do not occur in $(W_0,B_0)$, we form a new latin bitrade $(W_1,B_1)$ as
follows:
\begin{eqnarray*}
W_1 & = & W_0\cup 
\big\{(r,c'',s'),(r,c',s''),(r',c,s'),(r',c',s),(r_1,c',s'),(r_1,c'',s'')\big\}, \\
B_1& = & \big(B_0\setminus \{(r,c,s)\}\big)\cup 
\big\{(r,c,s'),(r,c',s),(r',c,s),(r',c',s'),(r,c'',s''), \\
& & (r_1,c'',s'),(r_1,c',s'')\big\}.
\end{eqnarray*} 
\fref{f:gen} shows the replacement for the black triangle $(r,c,s)$ 
used to form $(W_1,B_1)$. In this diagram the shaded triangles 
represent black triangles.
It is not hard to check that $(W_1,B_1)$ is connected and separated. 

\begin{figure}[h]
\centerline{\includegraphics[height=130pt]{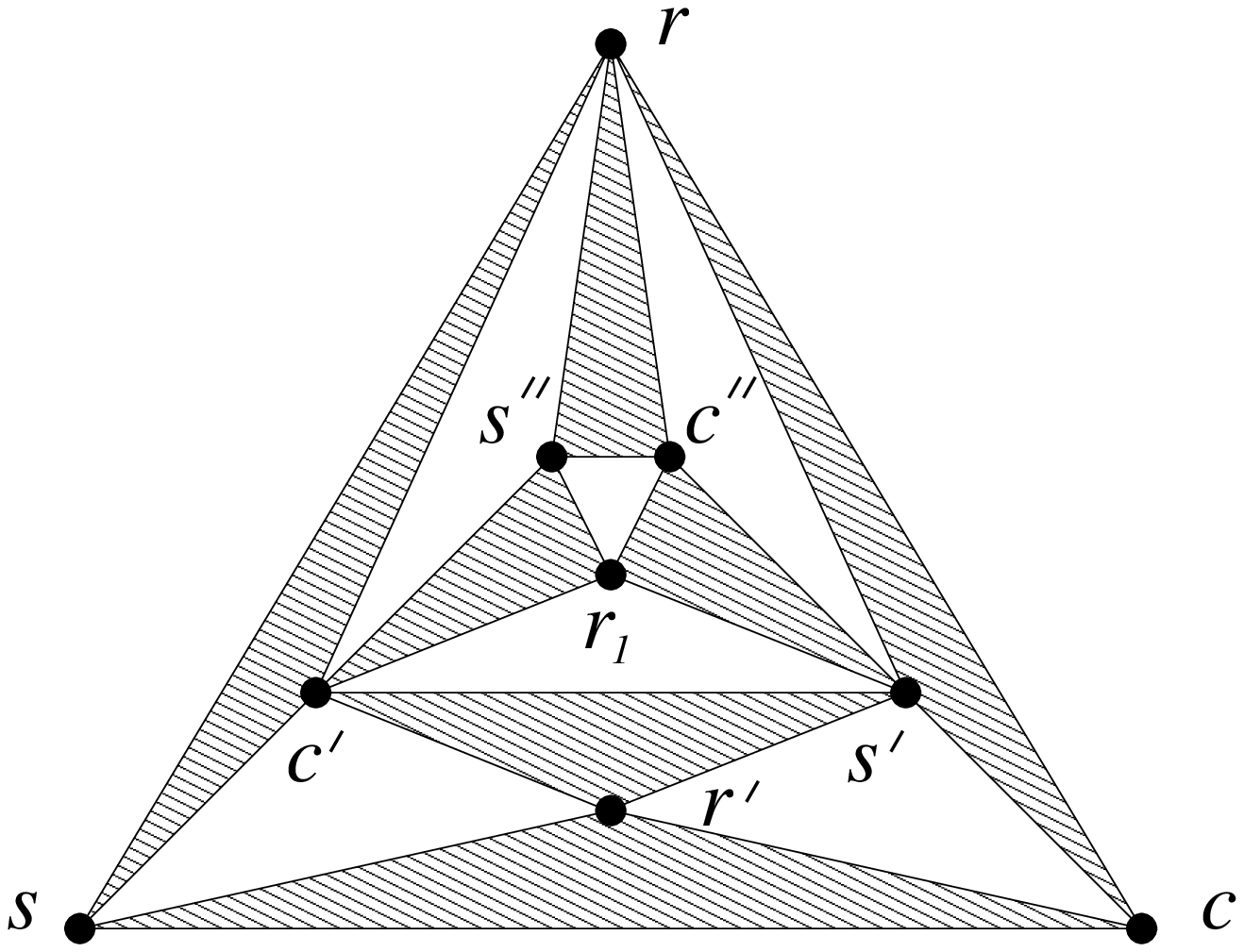}}
\caption{\label{f:gen} The transformation to $(W_1,B_1)$} 
\end{figure}

Moreover, it is clear from \fref{f:gen} that $(W_1,B_1)$ is still planar.
This fact can also be verified by \eref{e:genus},
as we have introduced two new rows, two new columns and
two new symbols, with $|W_1|=|W_0|+6$. 
Thus, $W_0$ embeds in a finite abelian group. For any such embedding:
\begin{eqnarray*}
r+c'' & = s' & = r_1 + c' \hbox{\rm \ and } \\
r + c' & = s'' & = r_1+c''.
\end{eqnarray*}
The above equations imply that $2r=2r_1$. 
 
By repeating this process for each of the $m$ black triangles from
$(W_0,B_0)$ that contain row $r$, we obtain a connected, separated
planar latin bitrade with size $8m$ and rows $r,r_1,\dots ,r_m$ that
must satisfy $2r=2r_1=\dots =2r_m$ in any group embedding.
\end{proof}

\begin{corollary}
For each integer $m\geq 2$, there exists a 
planar latin bitrade $(W,B)$ of size $8m$ such that if $W$
embeds in an abelian group $G$, then $G$ has torsion rank 
at least $\log_2(m+1)$. 
\end{corollary}

\begin{proof} 
For any fixed element $g$ in a finite abelian group $G$ of torsion
rank $\rho$, the number of solutions to $2x=g$ in $G$ is at most $2^\rho$.
This follows directly from the equivalent statement for cyclic groups.
In \lref{l:torank}, $r,r_1,\dots,r_m$ are $m+1$ distinct elements of $G$
such that $2r=2r_1=\dots =2r_m$.
\end{proof}

\section{Embedding Bitrades}\label{s:bitrades}

In earlier sections we have considered problems involving embedding
a latin trade in a group. However, for any latin trade there is at least
one disjoint mate, and it is natural to consider whether the embedding
behaviour of the mate is related to the embedding behaviour of the trade.
This is the issue considered in this section.

First we show that for spherical trades there is essentially only one
mate that needs to be considered.

\begin{lemma}\label{l:uniqmate}
Let $(W,B)$ be a planar latin bitrade. 
If $(W,B')$ is a latin bitrade and $B'\neq B$, then $(W,B')$ is 
neither connected nor separated.
\end{lemma}

\begin{proof}
Firstly, suppose that $(W,B')$ is separated but not connected.  Then,
there must exist a partition of the rows and columns of this latin
bitrade into non-empty sets $R_1,R_2$ and $C_1,C_2$, respectively,
such that each element of $B'$ (and hence $W$) lies either in a cell
from $R_1\times C_1$ or a cell from $R_2\times C_2$.  This implies
that $(W,B)$ is also disconnected, a contradiction.

Secondly, suppose that $(W,B')$ is connected but not separated. Then
we may transform $(W,B')$ into a separated latin bitrade of the same
size by increasing the number of rows, columns and symbols, as in
Example \ref{sepa}. This new latin bitrade is connected and separated
so has a well-defined genus $g$. However, from \eref{e:genus}, $g$
must be strictly less than 0, the genus of $(W,B)$, a contradiction.
 
Thirdly, suppose that $(W,B')$ is both connected and separated.  
The lemma is trivially true for intercalates, so we assume henceforth
there is a row, column or symbol which occurs at least three times.
By considering species equivalents if necessary, we
lose no generality by assuming that
there exists a row $r$ and some $j\geq 2$ such that:
$\{(r,c_i,s_i)\mid 0\leq i\leq j\}\subset W$, $\{(r,c_i,s_{i+1})\mid
0\leq i\leq j\}\subset B$ (where subscripts are calculated mod $j$),
but $(r,c_0,s_1)\not\in B'$.

We claim that such a circumstance implies that the planar graph $\G$ 
contains a subdivision of $K_{3,3}$, contradicting Kuratowski's theorem.

Consider the sequence $x(0),x(1),\dots ,x(\ell)$,
where $\ell\geq 2$, $x(0)=0$, $x(\ell)=1$ and
$(r,c_{x(i)},s_{x(i+1)})\in B'$ for $0\leq i\leq \ell-1$.
Note that such a sequence exists because $(W,B')$ is separated. 
Define $x(m)=\hbox{\rm{max}}\{x(i)\mid 1\leq i\leq \ell -1\}$. 
We now define our subdivision of $K_{3,3}$.  Let the vertices of one
colour be $r$, $c_0$ and $c_{x(m)}$ and the vertices of the other
colour be $s_0,s_1$ and $s_{x(m)}$.

Firstly, $r$ is clearly adjacent to $s_0$, $s_1$ and $s_{x(m)}$ in
$\G$.  Next, $(r,c_0,s_0)\in W$ implies that $c_0$ is adjacent to
$s_0$.  Similarly, $c_{x(m)}$ is adjacent to $s_{x(m)}$.  Also $c_0$
is adjacent to $s_1$ because $(r,c_0,s_1)\in B$.

Next, connect $c_{x(m)}$ to $s_0$ via a path on the sequence of vertices: 
$$c_{x(m)},s_{x(m)+1},c_{x(m)+1},\dots ,s_j,c_j,s_0.$$ 
Connect $c_0$ to $s_{x(m)}$ via the path 
$c_{x(0)},s_{x(1)},c_{x(1)},\dots ,s_{x(m)}$. 
Finally, connect $c_{x(m)}$ to $s_1$ via the path 
$c_{x(m)},s_{x(m+1)},c_{x(m+1)},\dots ,s_{x(\ell)}$. 
Observe that each of the above paths are disjoint on internal vertices. 
We thus have the required subdivision of $K_{3,3}$.
\end{proof}

It is important to realise that \lref{l:uniqmate} does not claim that
spherical trades have a unique disjoint mate. 
To illustrate this
point, we exhibit distinct latin bitrades
$(W,B)$ and $(W,B')$ such that $(W,B)$ is connected, separated and planar but 
$(W,B')$ is neither separated nor connected:
\[
W
=\begin{array}{|cccccc|}
\hline\nspacer
0&1&2&3&\cdot&\cdot\\[0.6ex]
4&5&\cdot&\cdot&3&2\\[0.6ex]
2&\cdot&0&\cdot&\cdot&\cdot\\[0.6ex]
\cdot&3&\cdot&1&\cdot&\cdot\\[0.6ex] 
3&\cdot&\cdot&\cdot&4&\cdot\\[0.6ex] 
\cdot&2&\cdot&\cdot&\cdot&5 
\sspacer\\
\hline
\end{array}
\quad
B
=\begin{array}{|cccccc|}
\hline\nspacer
3&2&0&1&\cdot&\cdot\\[0.6ex]
2&3&\cdot&\cdot&4&5\\[0.6ex]
0&\cdot&2&\cdot&\cdot&\cdot\\[0.6ex]
\cdot&1&\cdot&3&\cdot&\cdot\\[0.6ex] 
4&\cdot&\cdot&\cdot&3&\cdot\\[0.6ex] 
\cdot&5&\cdot&\cdot&\cdot&2 
\sspacer\\
\hline
\end{array}
\quad
B'
=\begin{array}{|cccccc|}
\hline\nspacer
2&3&0&1&\cdot&\cdot\\[0.6ex]
3&2&\cdot&\cdot&4&5\\[0.6ex]
0&\cdot&2&\cdot&\cdot&\cdot\\[0.6ex]
\cdot&1&\cdot&3&\cdot&\cdot\\[0.6ex] 
4&\cdot&\cdot&\cdot&3&\cdot\\[0.6ex] 
\cdot&5&\cdot&\cdot&\cdot&2 
\sspacer\\
\hline
\end{array}
\]

\begin{example}\label{eg:onegrp}
The smallest connected, separated bitrade $(W,B)$ such that $W$ embeds
in a group and $B$ does not embed in any group, is the following
example of genus $1$. Let $W$ be the bold entries in the left hand
matrix, and $B$ be the entries in the right hand matrix.
\[
\begin{array}{ccc}
\begin{array}{|cccccc|}
\hline\nspacer
{\bf 0}&{\bf 1}&{\bf 2}&3&{\bf 4}&5\\[0.6ex]
{\bf 1}&{\bf 2}&3&4&5&0\\[0.6ex]
{\bf 2}&3&{\bf 4}&5&{\bf 0}&1\\[0.6ex]
3&4&5&0&1&2\\[0.6ex]
{\bf 4}&5&{\bf 0}&1&{\bf 2}&3\\[0.6ex]
5&0&1&2&3&4\sspacer\\
\hline
\end{array}
&\qquad&
\begin{array}{|cccccc|}
\hline\nspacer
\looped{1}&2&\boxed{\looped{4}}&\cdot&\boxed{0}&\cdot\\[0.6ex]
2&1&\cdot&\cdot&\cdot&\cdot\\[0.6ex]
\looped{4}&\cdot&\looped{0}&\cdot&2&\cdot\\[0.6ex]
\cdot&\cdot&\cdot&\cdot&\cdot&\cdot\\[0.6ex]
0&\cdot&\boxed{2}&\cdot&\boxed{4}&\cdot\\[0.6ex]
\cdot&\cdot&\cdot&\cdot&\cdot&\cdot\sspacer\\
\hline
\end{array}
\\ \\
W\mbox{ embedded in }\Z_6&&B
\end{array}
\]
As shown, $W$ embeds in $\A^*=\Z_6$. However, $B$ clearly does not
embed in $\B^*=\Z_3$. Indeed, it cannot be embedded in any group since it
fails the quadrangle criterion, as demonstrated by the two quadrangles
marked with $\boxed{\phantom{0}}$ and $\,\looped{\phantom{00}}$.

To verify that no smaller bitrade has the
claimed property, it suffices to rule out spherical bitrades using
\tref{t:main} and note that the only two other bitrades of size 11 or
smaller are those discussed in \egref{eg:nogrp}.
\end{example}

\begin{example}\label{eg:diff}
For some spherical bitrades $(W,B)$,
the smallest group in which $W$ embeds is distinct
from the smallest group in which $B$ embeds. The smallest example of
this type has size 12. An example of that size is:
\[
\begin{array}{ccc}
\begin{array}{|cccccc|}
\hline\nspacer
\bf 0&\bf 1&\bf 2&\bf 3&4&5\\[0.6ex]
1&2&3&4&5&0\\[0.6ex]
\bf 2&\bf 3&\bf 4&5&0&\bf 1\\[0.6ex]
\bf 3&4&5&\bf 0&1&2\\[0.6ex]
4&5&0&1&2&3\\[0.6ex]
5&0&\bf 1&2&3&\bf 4\sspacer\\
\hline
\end{array}
&\qquad&
\begin{array}{|cccccc|}
\hline\nspacer
2&3&1&0&\cdot&\cdot\\[0.6ex]
\cdot&\cdot&\cdot&\cdot&\cdot&\cdot\\[0.6ex]
3&1&2&\cdot&\cdot&4\\[0.6ex]
0&\cdot&\cdot&3&\cdot&\cdot\\[0.6ex]
\cdot&\cdot&\cdot&\cdot&\cdot&\cdot\\[0.6ex]
\cdot&\cdot&4&\cdot&\cdot&1\sspacer\\
\hline
\end{array}
\\ \\
W\mbox{ embedded in }\Z_6&&B
\end{array}
\]
We have demonstrated that $W$ (the bold entries in the left hand
square) embeds in $\Z_6$. It does not embed in any smaller group,
since it uses five distinct columns (ruling out groups of order $<5$)
and contains a latin subsquare of order 2 (ruling out $\Z_5$).
With a little effort it is possible to show that neither $W$ or $B$
embeds in $S_3$. 

We now argue that $B$ cannot be embedded in a cyclic group of any
finite order $m$. Along the lines of \egref{eg:noncyc}, we assume
$r_0=c_0=s_0=0$ and deduce that $c_1=r_2=-c_2$ and $2c_1=c_2$, from
which we infer $3c_1=0$. Also $r_3=c_3=c_1/2$,
$r_5\circ c_5=c_2$ and $r_2\circ c_5=r_5\circ c_2$. These last
two equations imply $2r_5=r_2$. But now we have
$c_1=r_2=2r_3=2r_5$ and $c_1=4c_1=2c_2$ which implies that $c_2$,
$r_3$ and $r_5$ cannot be distinct. However, we know $r_3$ and $r_5$
are distinct and that $c_2\neq c_3=r_3$, so we conclude that $c_2=r_5$.
This is impossible, since then $r_5\circ c_5=c_2=r_5=r_5\circ c_0$,
which means that $c_5=c_0$.

We conclude that $B$ cannot embed in any cyclic group, whereas 
the smallest group that $W$ embeds in is cyclic. The bitrade $(W,B)$
is the smallest separated, connected example 
to achieve this property, and also the
smallest where the two trades have minimal group embeddings 
in distinct groups. (\egref{eg:noncyc} is the only smaller bitrade 
which does not have the property that both trades embed in the same
cyclic group.)
\end{example}

Despite the preceding examples which show that the embedding behaviour
is generally different for the two trades in a bitrade, it seems that
for spherical trades there is a special relationship between the two
trades.  If $(W,B)$ is a spherical latin bitrade, then so is $(B,W)$.
Thus we may apply the results of \sref{s:main} to the black triangles,
obtaining an embedding of $B$ into an abelian group $\B^*$.
Although the minimal embeddings of $W$ and $B$ may differ, as seen in
\egref{eg:diff}, we conjecture that their canonical embeddings cannot.

\begin{conjecture}\label{black=white}
If $(W,B)$ is a spherical latin bitrade, 
then $\A^*\iso\B^*$ and $\A\iso\B$.
\end{conjecture}
 
By \lref{l:AAstar}, the statements that $\A^*\iso\B^*$ and $\A\iso\B$
are equivalent.  \cjref{black=white} is true for all 72379 species of
(unordered) spherical latin bitrades of size up to 19, as catalogued
in \cite{tradenum}.  Moreover, the following result also supports
\cjref{black=white}.

\begin{theorem}\label{t:sameord}
If $(W,B)$ is a spherical latin bitrade then $|\A^*|=|\B^*|$.
\end{theorem}

\begin{proof}
To prove this, it suffices to show that $\per(\M^*)=\per(\MB^*)$.  Let
$D$ be the dual of the triangulation $\G$.  Let $X$ be the set of
matchings of size $|W|-1$ in $D$.  We count $|X|$ in two different
ways.  Note that $\per(\M^*)$ counts bijections between white faces
$f$ of $\G$ (other than the special white triangle $T_\Delta$) and
vertices $v_f$ of $\G$ (other than those of $T_\Delta$). We can
produce an element of $X$ by matching each $f$ with the black face
that is adjacent to $f$ and does not contain $v_f$. So, for each
vertex $v$ of $D$ corresponding to a white face $f$ of $\G$,
$\per(\M^*)$ is the number of $(|W|-1)$-matchings that do not cover
$v$ (this number is independent of our choice of $v$, by
\cref{c:allwhite}).  So $|X|=|W|\per(\M^*)$. Similarly, considering
the black triangles, $|X|=|B|\per(\MB^*)$. As $|W|=|B|$, the result
follows.
\end{proof}

\begin{figure}[h]
\centerline{(a)\lower70pt\hbox{\includegraphics[height=130pt]{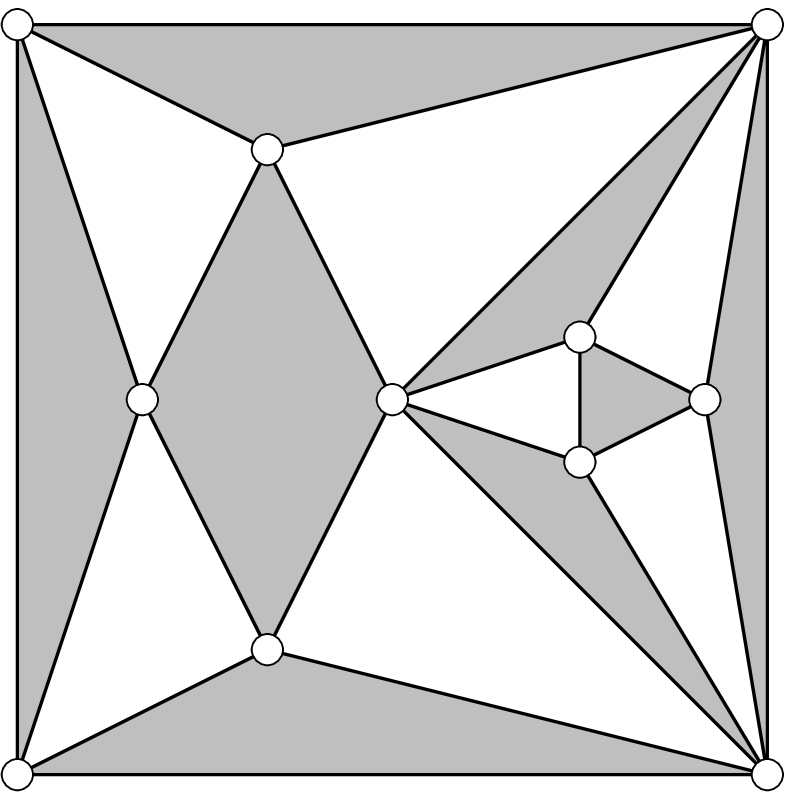}}
\qquad
(b) \lower70pt\hbox{\includegraphics[height=130pt]{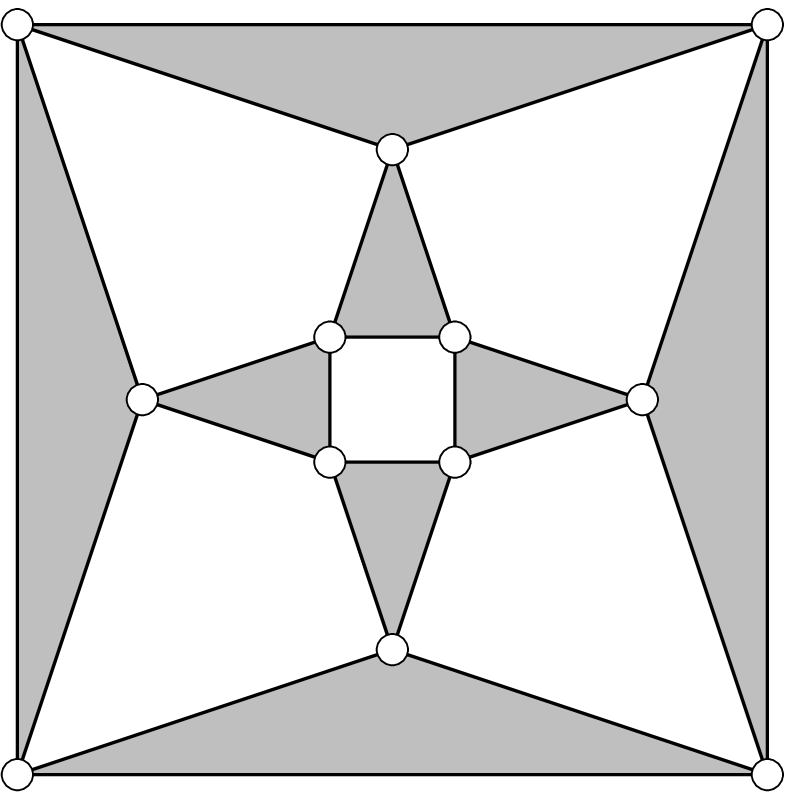}}}
\caption{\label{f:notri}Two face 2-colourable plane graphs}
\end{figure}

Several plausible generalisations of \cjref{black=white} do not hold.
\egref{eg:onegrp} shows that it does not hold for 2-colourable
triangulations of the torus. Also, if we stay in the plane but relax
the condition that faces must be triangular, the analogue of
\cjref{black=white} fails.
\fref{f:notri} shows two face 2-colourable plane graphs for which 
$\A\not\iso\B$. For graph (a), 
$\A\iso\Z_2\oplus\Z^3$ and $\B\iso\Z_4\oplus\Z^3$,
while for graph (b), $\A\iso\Z_2\oplus\Z^6$ and $\B\iso\Z^5$.

Also, we might drop the condition that indeterminates commute.  We
could label each vertex of $\G$ with a generator and take relations
which say that the product of the vertex labels in a clockwise
direction around each white triangle is the identity. We can still, on
the basis of \lref{l:000}, select a special white triangle and set its
labels to the identity. Suppose the group so produced is called $\NW$,
and the corresponding group for black triangles is $\NB$. It is not
always the case that $\NW\iso\NB$. A concrete example is the bitrade
in \egref{eg:noncyc}, for which $\NW=\<x,y\mid x^2=y^2,y^4=1\>$ and
$\NB=\<x,y\mid x^2=1,y^2xy^2=x\>$ (where we use 1 to denote the
identity).  Let $H_W$ and $H_B$ be the normal subgroups generated by the
squares of all the elements in $\NW$ and $\NB$ respectively. Then
$\NW/H_W\iso\NB/H_B\iso\<x,y\mid x^2=y^2=(xy)^2=1\>\iso\Z_2\oplus\Z_2$.
However, by a Reidemeister-Schreier rewriting process we find that
$H_W\iso\Z_2\oplus\Z\not\iso\Z\oplus\Z\iso H_B$, from which it follows
that $\NW\not\iso\NB$.

\section{Closing remarks}

Spherical latin bitrades are equivalent to 2-colourable
triangulations of the plane. We have shown (\tref{t:embeds}) that for
every spherical latin bitrade $(W,B)$ there is a copy of $W$ in a
finite abelian group $\A^*$. We showed that $\A^*$ is the torsion
subgroup of the group $\A$ defined by labelling the vertices of the
triangulation with commuting indeterminates and adding relations that
say the indeterminates around any white triangle add to the identity.

The presentation matrix $\M^*$ for $\A^*$ has the unusual
property that $\per(\M^*)=|\det(\M^*)|$ (\tref{t:detper}).  The
embedding of $W$ in $\A^*$ can be found in polynomial time.  Using
relations from the black triangles instead of white triangles produces
an abelian group $\B^*$. We have shown that $\A^*$ and $\B^*$ have the
same order (\tref{t:sameord}), and we conjecture (\cjref{black=white})
that they are isomorphic.

Aside from this conjecture a number of interesting open problems
remain.
\begin{itemize}
\item[Q1] Which abelian groups arise as $\A^*$ for some spherical
latin trade $W$? We saw in \sref{s:main} that every cyclic group does,
and so do at least some groups of large torsion rank. However
$\Z_2\oplus\Z_2$ does not, as can easily be checked by noting that it
contains no spherical trades other than intercalates.

\item[Q2] Can the spherical latin trades which embed in cyclic groups
be characterised by some property of their triangulations?

\item[Q3] Is there a family of trades $W$ for which $|\A^*|$ 
grows exponentially in $|W|$?

\item[Q4] What is the best algorithm for determining the minimal
embedding for a trade $W$ (that is, the abelian group of minimal order
such that $W$ can be embedded in the group)?

\item[Q5] For a trade $W$, how different can the minimal embedding
$E_W$ be from the canonical embedding in $\A^*$? Can $|\A^*|\big/|E_W|$
be arbitrarily large? Can the torsion rank for $\A^*$ be arbitrarily much
higher than the torsion rank of $E_W$? (We know from \egref{eg:diff}
that it can be higher.)

\item[Q6] Among the $(0,1)$ matrices of order $n$ for which the
permanent and determinant agree, what is the largest value that the
permanent can take and what structure of matrix achieves that
maximum? It might be helpful to restrict the class of matrices by placing
conditions on the row or column sums.  Answering such questions may
improve the upper bound in \cref{c:Afinite}, which is related to Q3.

\end{itemize}

There are many other interesting directions in which this research
could be taken. For example, one could (a)~look for other interesting
classes of trades/PLS that embed in groups, (b)~find new necessary or
sufficient conditions that govern whether a trade/PLS can embed in a
group, (c)~study groups defined on two face colourable plane graphs
other than triangulations, (d) consider infinite triangulations/trades
or (e)~investigate the relationship between
the groups $\NW$ and $\NB$ as introduced at the end of the previous
section (do they always have the same order?).  \egref{eg:torsion}
shows that $\NW$, $\NB$ will display behaviour that we have not seen
in the abelian case.

  \let\oldthebibliography=\thebibliography
  \let\endoldthebibliography=\endthebibliography
  \renewenvironment{thebibliography}[1]{%
    \begin{oldthebibliography}{#1}%
      \setlength{\parskip}{0.2ex}%
      \setlength{\itemsep}{0.2ex}%
  }%
  {%
    \end{oldthebibliography}%
  }

\subsection*{Acknowledgements}

The authors are very grateful to Iain Aitchison for interesting
discussions, in particular, the tantalising suggestion of connections
between \cjref{black=white} and knot theory. We are also indebted to
the members of the online group-pub-forum, and particularly to Mike
Newman, for help with manipulating presentations for non-abelian groups.

We are also very grateful to Ale\v s Dr\'apal for alerting us after
submission to some important and highly relevant work that we had
overlooked. Our Lemmas \ref{l:000}, \ref{l:dropZ} and \ref{l:AAstar}
can be deduced from results in \cite{DrKe2} and \cite{DrKe3}, although
the terminology of those papers is somewhat different to the present
work.  Lemma~4.4 of \cite{DrKe4} proves an exponential bound similar
to our \cref{c:Afinite}, but with a better base constant. Indeed,
\cite{DrKe2,DrKe3,DrKe4} anticipate several of the pivotal ideas in
our present investigation, including using entries of a partial latin
square to create relations in a group, using those relations to write
down a $(0,1)$-matrix and investigating conditions under which that
matrix has non-zero determinant.  Finally, it seems that in \cite{DHK}
our \tref{t:main} has been obtained independently using a quite
different, and more geometrical argument.

\end{document}